\documentclass[]{article}  
\usepackage{CJK}
\usepackage{amssymb}
\usepackage{bbm}
\usepackage{amsfonts}
\usepackage{mathrsfs}
\usepackage{graphicx}
\usepackage{amsmath}
\usepackage{amsthm}
\usepackage{xypic}
\usepackage{graphicx}
\usepackage{times}
\usepackage{geometry}
\usepackage{color}
\usepackage{indentfirst}
\usepackage{cases}
\usepackage{hyperref}
\usepackage{graphicx}
\usepackage{float}
\usepackage{amsthm}
\usepackage{amssymb}
\usepackage{amsmath}
\usepackage{mathrsfs}
\usepackage{indentfirst}
\usepackage{geometry}
\usepackage{authblk}
\usepackage{hyperref}
\usepackage{enumerate}
\usepackage[numbers,sort&compress]{natbib}

\newtheorem{Theorem}{Theorem}[section]
\theoremstyle{definition}

\newtheorem{Lemma}[Theorem]{Lemma}

\theoremstyle{definition}

\title{The multiplicity of the laplacian eigenvalue 1 of a tree}
\author{Songnian Xu \thanks{Corresponding author. E-mail address: xsn131819@163.com},\ \ \ Dein Wong \thanks{Corresponding author. E-mail address:wongdein@163.com.  Supported by the National Natural Science Foundation of China(No.12371025)} ,\ \ Chi Zhang \thanks{ Supported by NSFC of China(No.12001526)}, \ \ Wenhao Zhen
\\ {\small  \it School of Mathematics, China University of Mining and Technology, Xuzhou,  China.}

}
\date{}

\begin{document}
\baselineskip 17pt

\title{The multiplicity of the laplacian eigenvalue 1 of a tree}

\author{Songnian Xu \thanks{Corresponding author. E-mail address: xsn131819@163.com},\ \ \ Dein Wong \thanks{Corresponding author. E-mail address:wongdein@163.com.  Supported by the National Natural Science Foundation of China(No.12371025)} ,\ \ Wenhao Zhen
\\ {\small  \it School of Mathematics, China University of Mining and Technology, Xuzhou,  China.}

}

\date{}
\maketitle

\begin{abstract}
\ \ Let $G$ be a connected, undirected simple graph. Denote by $L(G)$ the Laplacian matrix of $G$, and let $m_{G}(\lambda)$ be the multiplicity of an eigenvalue $\lambda$ of $L(G)$. When $G$ is a tree $T$ with $n \ge 6$ vertices, Tian et al. [Discrete Mathematics, 2026] proved that if $T$ is reduced and contains no pendant $P_3$, then
\[
m_{T}(1) \le \frac{n-6}{4},
\]
and they gave a complete characterization of the graphs for which equality holds.

In this paper, we further investigate the above problem. Still assuming that $T$ is a tree with $n \ge 7$ vertices which is reduced and has no pendant $P_3$, we prove the following results. If $m_T(1) \neq \frac{n-6}{4}$, then
\[
m_{T}(1) \le \frac{n-7}{4},
\]
and we give a complete characterization of the graphs for which equality holds. If, moreover, $m_T(1) \neq \frac{n-6}{4}, \frac{n-7}{4}$, then
\[
m_{T}(1) \le \frac{n-8}{4},
\]
and we also give a complete characterization of the extremal graphs.
\end{abstract}

\vskip 2.5mm
\noindent{\bf AMS classification:} 05C50
 \vskip 2.5mm
 \noindent{\bf Keywords:}   Laplacian eigenvalue; trees; eigenvalue multiplcity

\section{Introduction}
In this paper, we consider a simple undirected connected graph $G$. Let $V(G)$ and $E(G)$ denote the vertex set and edge set of $G$, respectively. For $V_1 \subseteq V(G)$ and $E_1 \subseteq E(G)$, $G - V_1$ denotes the induced subgraph of $G$ on $V(G) \setminus V_1$, and $G - E_1$ denotes the graph obtained from $G$ by deleting the edges in $E_1$. For $v \in V(G)$, let $d_v(G)$ denote the degree of $v$ in $G$. When no confusion arises, we simply write $d_v$ or $d(v)$. In a graph $G$, $d(u,v)$ denotes the length of a shortest path between $u$ and $v$. The diameter of $G$ is $d(G) = \max\{d(u,v) \mid u, v \in V\}$. Let $A(G) = (a_{ij})$ denote the adjacency matrix of $G$, where $a_{ij}=1$ if vertices $i$ and $j$ are adjacent, and $a_{ij}=0$ otherwise. The Laplacian matrix of $G$ is $L(G) = D(G) - A(G)$, where $D(G)$ is the diagonal matrix whose diagonal entries are the degrees of the vertices of $G$. The multiplicity of an eigenvalue $\lambda$ of $L(G)$ is denoted by $m_G(\lambda)$. 
Moreover, if \( G \) is obtained from a graph \( G_1 \) and a path \( P_k \) by joining a pendant vertex of \( P_k \) and an arbitrary vertex of \( G_1 \), then we say \( P_k \) is a pendant path of \( G \). For convenience, a pendant path on \( k \) vertices in a graph is denoted by \( \mathcal{P}^k \).

In 1985, Faria \cite{Fa} proved that for any graph $G$, $m_G(1) \ge p(G) - q(G)$, where $p(G)$ denotes the number of pendant vertices of $G$ and $q(G)$ denotes the number of quasi-pendant vertices of $G$. Andrada et al. \cite{EA} gave a partial characterization of graphs satisfying $m_G(1) = p(G) - q(G)$. Later, Grone et al. \cite{RG} proved that for any tree $T$, $m_T(1) \le p(T) - 1$. Gupta \cite{Gu} showed that $m_T(1) = p(T) - 1$ if and only if $d(u_1, u_2) \equiv 2 \pmod{3}$ for any two pendant vertices $u_1, u_2$ of $T$. Subsequently, in 2026, Wong et al. \cite{Dein} gave a complete characterization of graphs satisfying $m_T(\lambda) = p(T) - 1$. Guo et al. \cite{JM1} showed that if $T$ is a tree with $n$ vertices, then $m_T(1) \in S = \{0, 1, \dots, n-5, n-4, n-2\}$ and for every $k \in S$, there exists a tree $T$ with $n$ vertices such that $m_T(1) = k$. In 2025, Wang et al. \cite{Wang1} showed that if $T$ is a tree with $n$ vertices and $\beta'(T) \ge 2$, then $m_T(1) \le n - 2\beta'(T) - 1$, where $\beta'(T)$ is the induced matching number of $T$, and they gave a complete characterization of extremal graphs. Li et al. \cite{Li} proved that $m_G(1) \le 2c(G) + P(G)$, where $c(G)$ denotes the dimension of the cycle space of $G$, and equality holds if and only if $G = C_n$ with $6 \mid n$.

A graph $G$ is said to be reduced if $p(G) = q(G)$. In 2025, Tian et al. \cite{Tian2} proved that if a tree $T$ is reduced, then $m_T(1) \le \frac{n-2}{4}$, and they gave a complete characterization of extremal graphs, as stated in Theorem 1.2 below. Subsequently, in 2026, Tian \cite{Tian2} further investigated this problem and proved that if $T$ is reduced and contains no pendant $P_3$, then $m_T(1) \le \frac{n-6}{4}$, and equality holds if and only if $T \cong G_1$ (see Fig. 1). In this paper, we continue the study of Tian's results. We prove that if $T$ is reduced, contains no pendant $P_3$, and $T \not\cong G_1$, then $m_T(1) \le \frac{n-7}{4}$, and the extremal graphs are exactly $G_4$ or $G_5$ (see Fig. 2). Moreover, if we further require that $T$ is reduced, contains no pendant $P_3$, and $T \not\cong G_1, G_4, G_5$, then $m_T(1) \le \frac{n-8}{4}$, and the extremal graphs are exactly $G_2$ or $G_3$ (see Fig. 1). The main results of this paper are stated as Theorems 1.3 and 1.4 below.

\begin{figure}[H]
  \centering
  \includegraphics[width=0.8\linewidth]{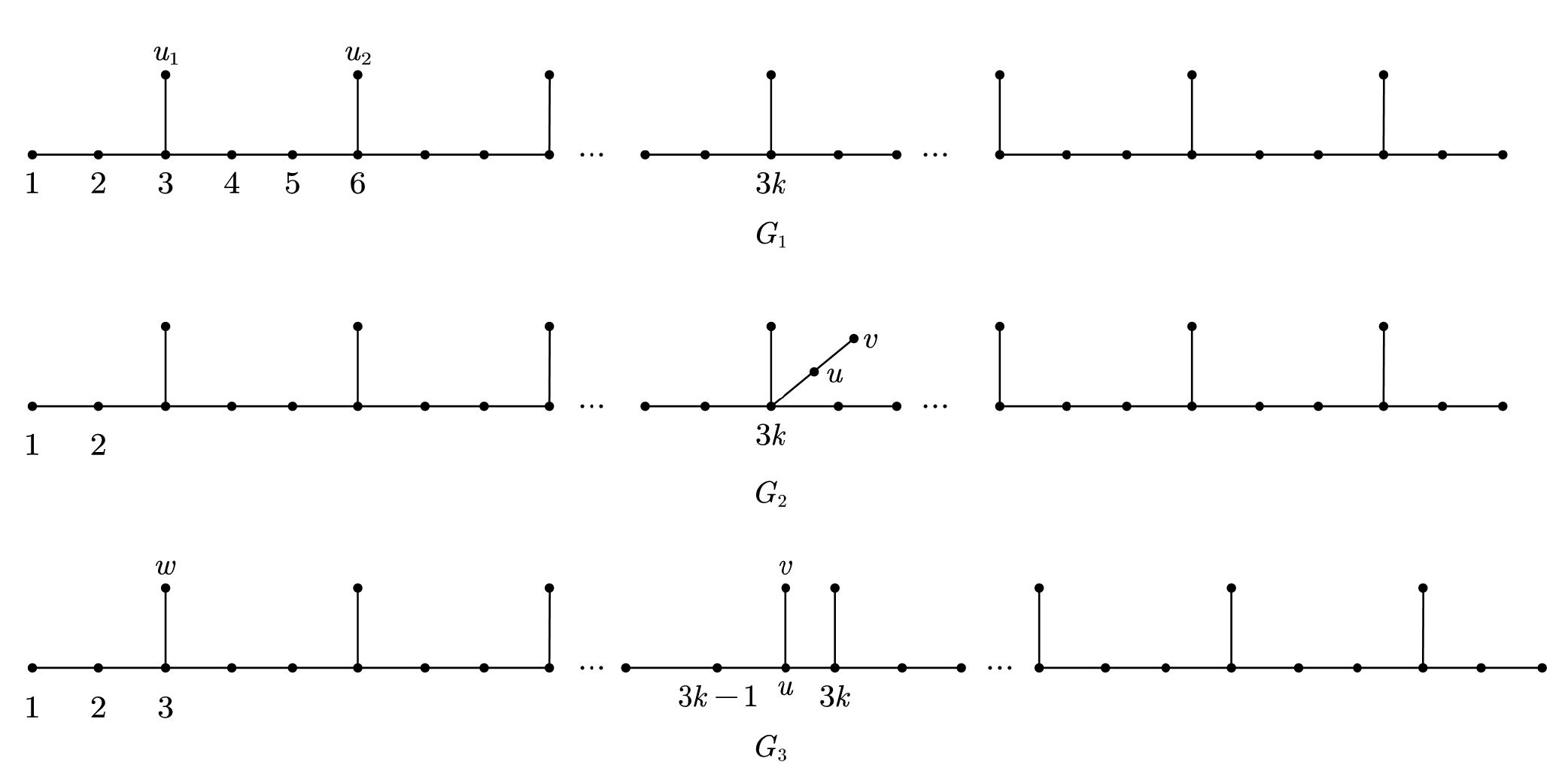}
   \caption{$G_1$,$G_2$ and $G_3$} 
\end{figure}
\begin{Theorem}\cite{Tian1}
Let $T$ be a reduced tree without $\mathcal{P}^3$ and let the order of $T$ be $n \ge 6$. Then we obtain that
\[
m_{T}(1) \leq \frac{n - 6}{4},
\]
and the equality holds if and only if $T\cong G_1$ and $n \equiv 2 \pmod{4}$.
\end{Theorem}

As shown in Fig. 1, $G_2$ and $G_3$ are obtained from $G_1$ by adding a  path $P_2$. More precisely, $G_2$ is obtained from $G_1$ by attaching a pendant $P_2$ at $v_{3k}$, while $G_3$ is obtained from $G_1$ by inserting a vertex $u$ between $v_{3k-1}$ and $v_{3k}$ (or between $v_{3k}$ and $v_{3k+1}$; by symmetry the two cases are equivalent) and then attaching a pendant vertex $v$ to $u$.

Given a graph $G$, pick a pendant vertex $u$ adjacent to $v$, and let the neighbors of $v$ other than $u$ be $v_1, v_2, \dots, v_s$. Remove $u$ and $v$, then attach a copy of $P_2$ to each $v_i$ by joining one endpoint of the $P_2$ to $v_i$. The resulting graph is denoted $G'$ and referred to as the \textit{reduction graph} of $G$ .
If $v$ has degree $2$ in $G$, then $G'$ coincides with $G$. This operation may be applied repeatedly until no quasi-pendant vertex of degree at least $3$ remains. The final outcome is called the \textit{final reduction graph} of $G$.
\begin{Theorem}\cite{Tian2}
Let $T$ be a reduced tree on $n \ge 6$ vertices. Let $T'_i$ ($i = 1, 2, \dots, k$) be the components of the final reduction graph of $T$. Then
\[
m_{T}(1) \le \frac{n - 2}{4},
\]
and the equality holds if and only if the following two assertions hold:
\begin{enumerate}
\item[(i)] Each vertex of degree greater than $2$ is a quasi-pendant vertex in $T$;
\item[(ii)] Each $T'_i$ of $T$ is a path $P_6$.
\end{enumerate}
\end{Theorem}

\begin{Theorem}
Let $T$ be a reduced tree without $\mathcal{P}^3$ and $|T| = n \ge 7$. If $T \not\cong G_1, G_4$ or $G_5$, then
\[
m_{T}(1) \le \frac{n-8}{4},
\]
and equality holds if and only if $T \cong G_2$ or $G_3$ with $n \equiv 0 \pmod{4}$.

\end{Theorem}

\begin{figure}[H]
  \centering
  \includegraphics[width=0.7\linewidth]{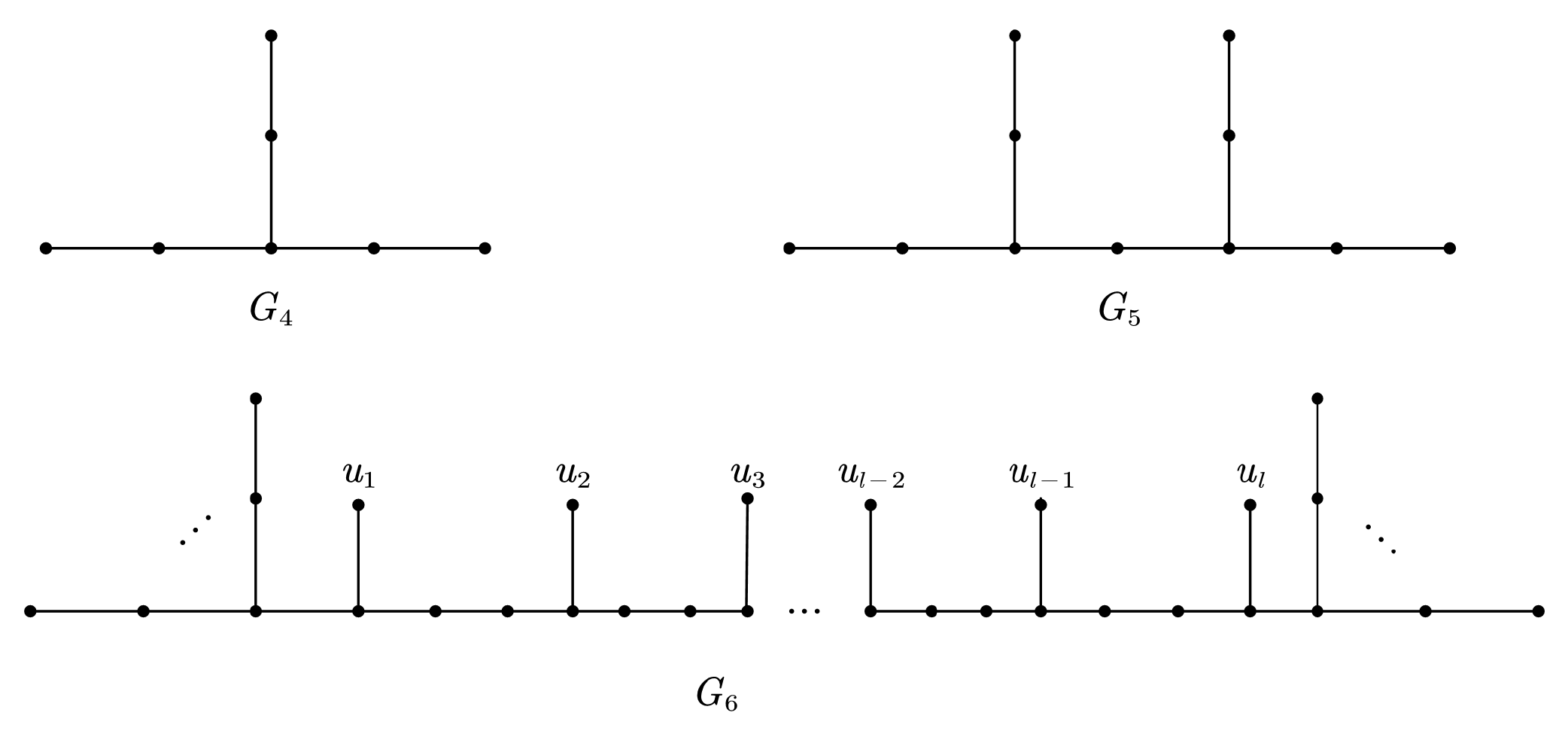}
   \caption{$m_{G_4}(1)=0$, $m_{G_5}(1)=1$ and $m_{G_6}(1)=l-1$} 
\end{figure}

\begin{Theorem}
Let $T$ be a reduced tree without $\mathcal{P}^3$ and $|T| = n \ge 7$. If $T \not\cong G_1$, then
\[
m_{T}(1) \le \frac{n-7}{4},
\]
and equality holds if and only if $T \cong G_4$ or $G_5$.
\end{Theorem}

\section{Proof of Theorems}

\begin{Lemma}\cite{MB}
Let $G - e$ be the graph obtained from a graph $G$ by deleting an edge $e$. Then
\[
-1 + m_{G-e}(1) \leq m_{G}(1) \leq 1 + m_{G-e}(1).
\]
\end{Lemma}

\begin{Lemma}\cite{RG}
Let $T$ be a tree obtained from a tree $T_1$ and a path $P_3$ by joining a vertex of $T_1$ to a pendant vertex of $P_3$. Then it follows that
\[
m_{T}(1) = m_{T_1}(1).
\]
\end{Lemma}

\begin{Lemma}\cite{Wang1}
Let $G'$ be a graph obtained by adding a pendant vertex $w$ to a quasipendant vertex $u$ of the graph $G$ of order $n$, then
\[
m_{G'}(1) = m_G(1) + 1.
\]
\end{Lemma}

\begin{figure}[H]
  \centering
  \includegraphics[width=1.0\linewidth]{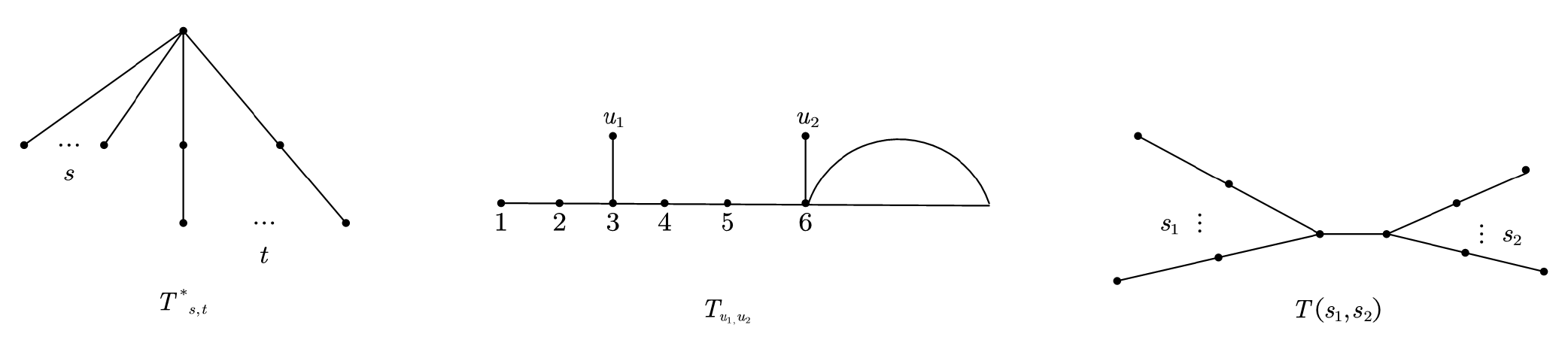}
   \caption{$T_n^*(s,t)$, $T_{u_1,u_2}$ and $T(s_1,s_2)$} 
\end{figure}

\begin{Lemma}
Let $T_n^*(s,t)$ and $T(s_1,s_2)$ be as shown in Fig. 3, where $s + t \ge 1$, $|T_n^*(s,t)| \neq 3$, $s_1 \ge 2$ and $s_2 \ge 2$. Then

(1) \cite{JM1}
\[
m_{T_n^*(s,t)}(1) =
\begin{cases} 
s - 1, & s \ge 1, \\
0,     & s = 0;
\end{cases}
\]

(2) \cite[2.10]{Tian1}
\[
m_{T(s_1,s_2)}(1) = 0.
\] 

\end{Lemma}

\begin{Lemma}\cite{Tian1}
Let $T$ be a tree with an internal path $P_5 = u_1u_2u_3u_4u_5$. Let $H$ be the tree obtained from $T$ by deleting the vertices $\{u_2, u_3, u_4\}$ and joining $u_1$ and $u_5$ with an edge. Then $m_{T}(1) = m_{H}(1)$.
\end{Lemma}

\begin{Lemma}\cite{Tian1}
Let $T\cong T_{u_1,u_2}$, see Fig. 3, which has a local structure induced by $\{v_1, v_2, \cdots, v_6,u_1,u_2\}$. Let $T_1$ be the tree obtained from $T$ by deleting $\{v_1, v_2, v_3, u_1\}$. Then we get $m_{T}(1) = 1 + m_{T_1}(1)$.
\end{Lemma}

\begin{Lemma}\cite{Tian1}
Let $G$ be a graph with a pendant vertex $u$ and $u \sim v$. Suppose that $d_v \geq 3$ and $w$, distinct from $u$, is a neighbor of $v$. Let
$P_2 = xy$ be a path on two vertices. Denote by $\Gamma$ the graph obtained from $G - e_{vw}$ and $P_2 = xy$ by joining $w$ and $y$. Then we have
\[
m_{G}(1) = m_{\Gamma}(1).
\]
\end{Lemma}

From Lemma 2.7 above, we can quickly obtain the following Lemma 2.8.

\begin{Lemma}
Let $G$ be a graph, $u$ a pendant vertex of $G$, and $v$ a quasipendant vertex of $G$ with $d_v \ge 2$. If a pendant path $P_2$ is attached to $v$ to obtain the graph $\Gamma$, then
\[
m_{G}(1) = m_{\Gamma}(1).
\]
\end{Lemma}

\begin{proof}
By Lemma 2.7, we have
\[
m_{\Gamma}(1) = m_{G}(1) + m_{P_4}(1) = m_{G}(1).
\]
\end{proof}

\begin{figure}[H]
  \centering
  \includegraphics[width=1.0\linewidth]{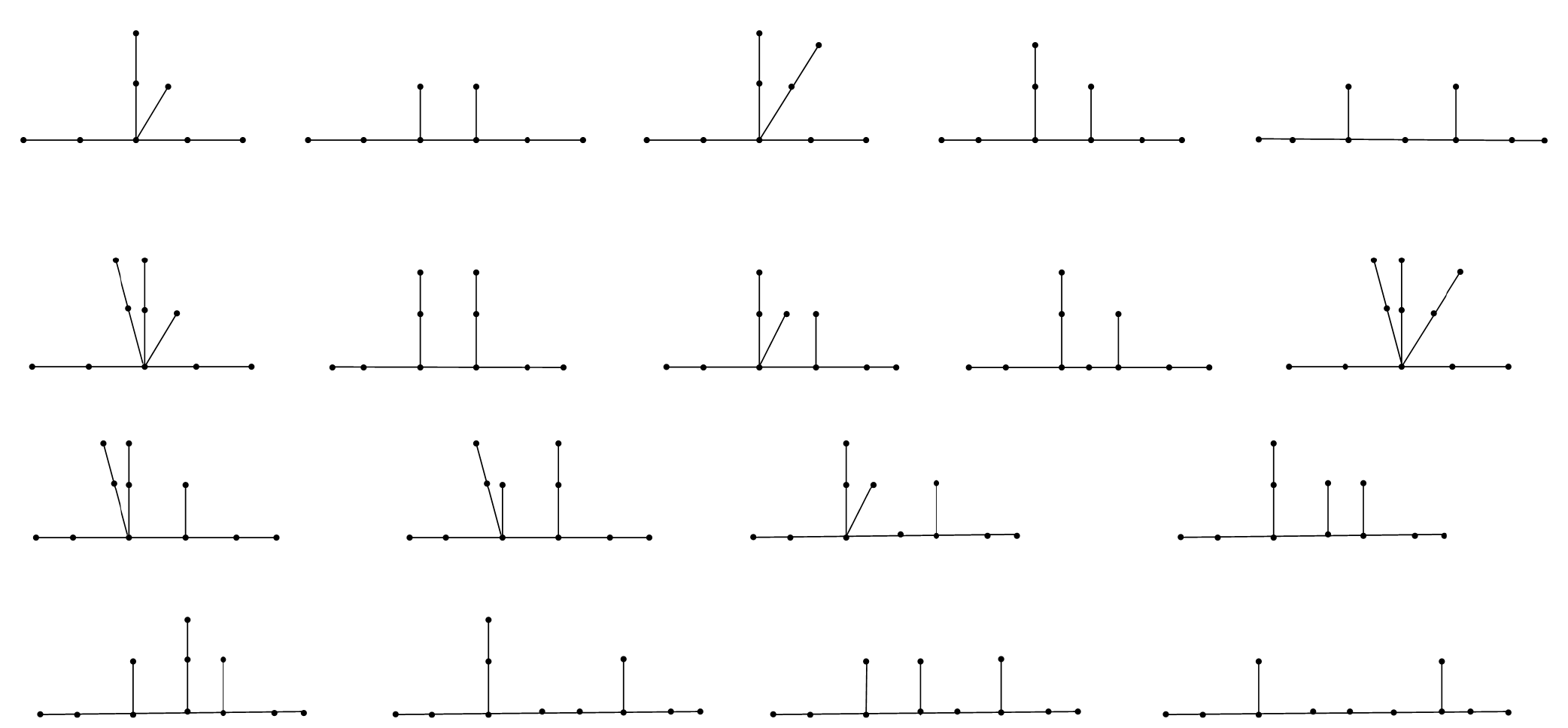}
   \caption{$m_{T}(1)=0$} 
\end{figure}

\begin{Lemma}
If $T$ is any tree in Fig. 4, then $m_{T}(1) = 0$.
\end{Lemma}

\begin{proof}
We only consider the case where $T$ is the first tree in Fig. 4; the other cases can either be proved similarly or follow directly from Lemma 2.4.
By Lemma 2.8, deleting any pendant path $P_2$ attached to a quasipendant vertex does not change the multiplicity of the Laplacian eigenvalue $1$ of the graph. Therefore,
\[
m_{T}(1) = m_{P_4}(1) = 0.
\]

\end{proof}

\begin{proof}[Proof of Theorem 1.3]

We first prove the sufficiency, i.e., if $T \cong G_2$ or $G_3$, then $m_{T}(1) = \frac{n-8}{4}$.

When $T \cong G_2$, by Lemma 2.8 we have $m_{T}(1) = m_{T-\{u,v\}}(1)$. Note that $T-\{u,v\} \cong G_1$, so by Theorem 1.1,
\[
m_{T}(1) = m_{T-\{u,v\}}(1) = \frac{n(T-\{u,v\})-6}{4} = \frac{n-8}{4}.
\]

If $T \cong G_3$, we proceed by induction on $n$. When $n = 8$, $T$ is the second graph in Fig. 4, so $m_{T}(1) = 0 = \frac{n-8}{4}$, and the conclusion holds. Now assume $n \ge 12$ and that the conclusion holds for all trees of order less than $n$. By the symmetry of $T \cong G_3$, we may assume $k \ge 2$ (where $k$ is as marked in $G_3$). Then by Lemma 2.6,
\[
m_{T}(1) = 1 + m_{T-\{v_1, v_2, v_3, w\}}(1).
\]
Note that $T-\{v_1, v_2, v_3, w\}$ is again of the form $G_3$. By the induction hypothesis,
\[
m_{T-\{v_1, v_2, v_3, w\}}(1) = \frac{n' - 8}{4},
\]
where $n' = |T-\{v_1, v_2, v_3, w\}|$. Hence
\[
m_{T}(1) = 1 + m_{T-\{v_1, v_2, v_3, w\}}(1) = 1 + \frac{n' - 8}{4} = \frac{n-8}{4}.
\]

Now we prove the necessity.

When $7 \le n \le 11$, Fig. 4 shows all graphs satisfying the condition.
Note that when $n = 7$, $T$ can only be of the form $G_4$. When $n = 8$, $T$ can only be either the first or the second graph in the first row of Fig. 4, which are exactly $G_2$ and $G_3$, respectively. Therefore, for $7 \le n \le 11$, the conclusion holds by Lemma 2.9. So we assume $n \ge 12$ and that the conclusion holds for all trees $T'$ with $7 \le |T'| < n$.

Let $v_1 \sim v_2 \sim v_3 \sim \dots \sim v_{d-1} \sim v_d \sim v_{d+1}$ be a diametral path of $T$. Since $T$ contains no pendant $P_3$, both $v_3$ and $v_{d-1}$ have degree at least $3$. This fact will be used throughout the following discussion without further mention.
\medskip

\textbf{Case 1.} $v_3, v_{d-1}$ are not quasi-pendant vertices.

If $d = 4$, then $T$ is a $P_2$-star. By Lemma 2.4 we have $m_{T}(1)=0 < \frac{n-8}{4}$. Hence we assume $d \ge 5$.

Let $T - e_{3,4} = T_3 \cup T_4$, where $T_3$ is the component containing $v_3$ and $T_4$ is the component containing $v_4$.

\medskip

\begin{figure}[H]
  \centering
  \includegraphics[width=1.0\linewidth]{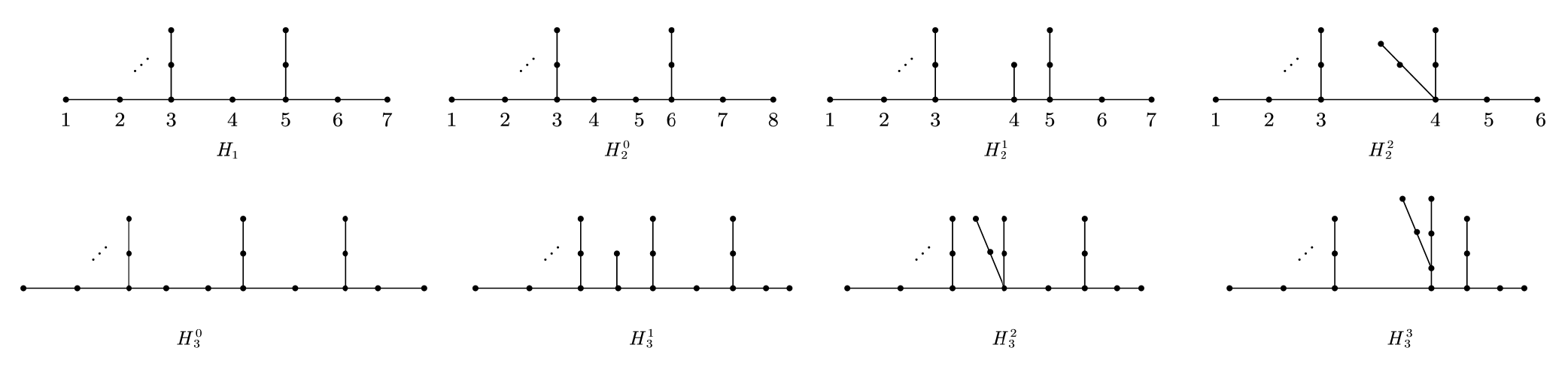}
   \caption{$H_1$, $H^{i}_2$ and $H^{j}_3$ } 
\end{figure}

\textbf{Case 1.1.} $T_4$ contains no pendant $P_3$ and $P(T_4)=q(T_4)$.

Clearly $|T_4| \ge 6$. If $|T_4| = 6$, then $T \cong H_1$. Note $T \not\cong G_5$, so $n > 13$. By Lemma 2.1,
$m_{T}(1) \le 1 + m_{T_3}(1) + m_{T_4}(1)$.
From Lemma 2.4, we know that \( m_{T_3}(1) = m_{T_4}(1) = 0 \). Therefore,
$m_{T}(1) \le 1 < \frac{n-8}{4}$.

If $|T_4| \ge 7$ and $T_4 \cong G_4$ or $G_5$, then $T \cong H^0_2$,  $H^1_2$,  $H^2_2$,  $H^0_3$, $H^1_3$,  $H^2_3$ or $H^3_3$. 
We only discuss the case $T \cong H^0_2$; the other cases can be proved similarly and yield the same conclusion.
If $T \cong H^0_2$ and $d_{v_3}=3$, then by computer calculation $m_{T}(1)=0 < \frac{n-8}{4}$. If $d_{v_3} > 3$, then $|T| \ge 14$. Clearly $T_3, T_4$ are both of the form $T_n^*(s,t)$. By Lemma 2.4, $m_{T_3}(1)=m_{T_4}(1)=0$. Hence
\[
m_{T}(1) \le 1 + m_{T_3}(1) + m_{T_4}(1) = 1 < \frac{n-8}{4}.
\]

If $|T_4| \ge 7$ and $T_4 \not\cong G_4, G_5$, then by the induction hypothesis,
\[
m_{T_4}(1) \le \frac{n(T_4)-8}{4},
\]
so
\[
m_{T}(1) \le \frac{n(T_4)-8}{4} + 1 = \frac{n-9}{4} < \frac{n-8}{4}.
\]

\medskip

\begin{figure}[H]
  \centering
  \includegraphics[width=1.0\linewidth]{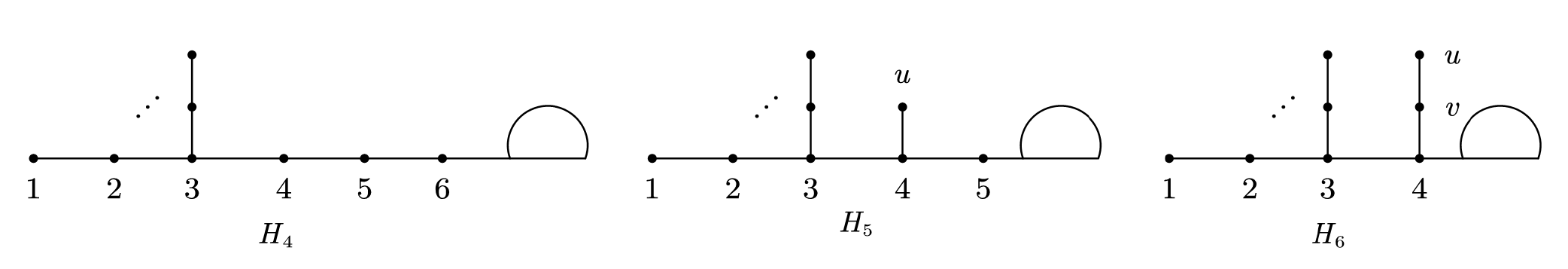}
   \caption{$H_4$, $H_5$ and $H_6$ } 
\end{figure}
\textbf{Case 1.2.} $T_4$ contains a pendant $P_3$.

Then $T$ is isomorphic to one of $H_4, H_5, H_6$.

\medskip

\textbf{Case 1.2.1.} If $G \cong H_4$, let $T'$ be the graph obtained by deleting $v_4, v_5, v_6$ and then joining $v_3$ and $v_7$. Clearly $T'$ contains no pendant $P_3$ and is reduced. By Theorem 1.1 and Lemma 2.5,
\[
m_{T}(1) = m_{T'}(1) < \frac{n(T')-6}{4} \le \frac{n-3-6}{4} < \frac{n-8}{4},
\]
where the inequality $m_{T'}(1) < \frac{n(T')-6}{4}$ holds because $v_{d-1}$ is not a quasi-pendant vertex.
\medskip

\textbf{Case 1.2.2.} If $G \cong H_5$, let $T'_4 = T_4 - \{v_4, v_5, u\}$.

Applying Lemma 2.7 to $v_4$, we have:
\[
m_{T}(1) = m_{T''}(1) + m_{T_4}(1) = m_{T_4}(1),
\]
where $T''$ is the $P_2$-star obtained by attaching a pendant $P_2$ to $v_3$ of $T_3$. By Lemma 2.2, $m_T(1)=m_{T_4}(1) = m_{T'_4}(1)$.
Next, we examine $T'_4$.

\begin{figure}[H]
  \centering
  \includegraphics[width=1.0\linewidth]{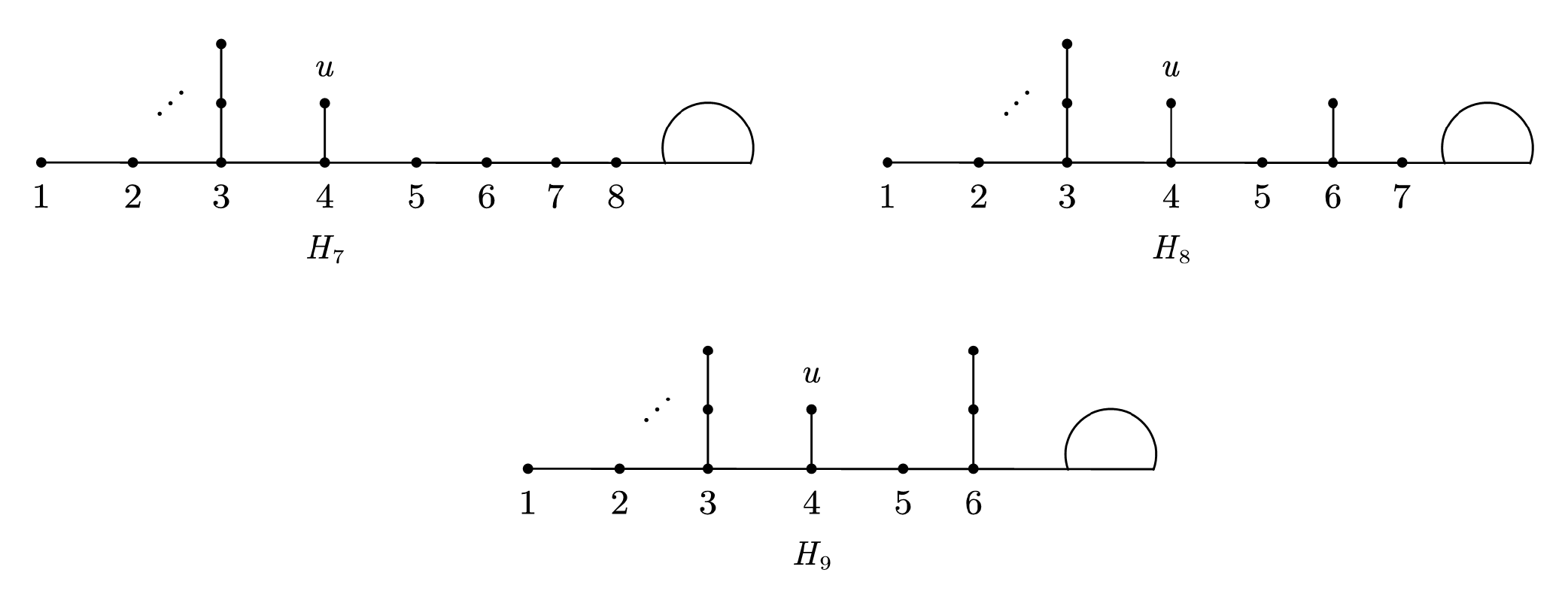}
   \caption{$H_7$, $H_8$ and $H_9$ } 
\end{figure}

\textbf{Case 1.2.2.1.} If $T'_4$ contains a pendant $P_3$, then $T$ is isomorphic to $H_7, H_8$ or $H_9$:
\begin{itemize}
\item If $T \cong H_7$, let $T'$ be obtained by deleting $\{5,6,7\}$ and joining $v_4$ and $v_8$. As in Case 1.2.1, we get $m_{T}(1) < \frac{n-8}{4}$.
\item If $T \cong H_8$, then $|T'_4| \ge 8$ and $T'_4$ is reduced. 
By Theorem 1.2,
\[
m_{T}(1) = m_{T'_4}(1) < \frac{n(T'_4)-2}{4} \le \frac{n-8-2}{4} < \frac{n-8}{4}.
\]

\item If $T \cong H_9$, a similar discussion as for $H_8$ leads to the same conclusion.
\end{itemize}

\begin{figure}[H]
  \centering
  \includegraphics[width=0.6\linewidth]{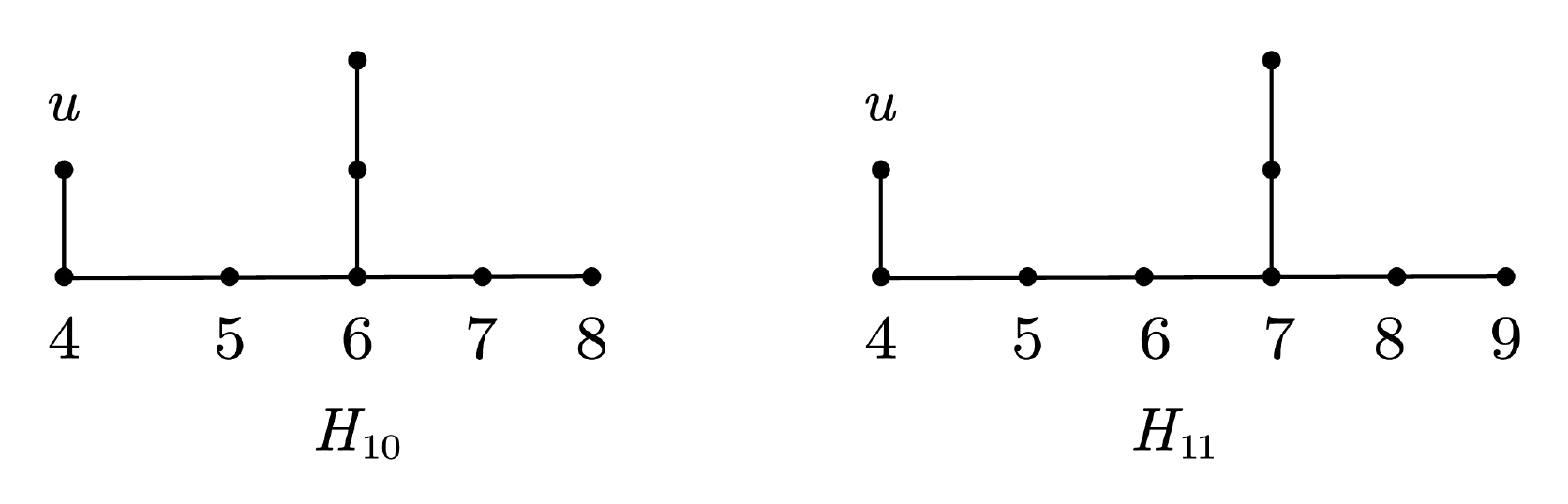}
   \caption{$H_{10}$ and $H_{11}$ } 
\end{figure}

\textbf{Case 1.2.2.2.} If $T'_4$ contains no pendant $P_3$ and $T'_4$ is reduced, we have $|T'_4| \ge 5$:
\begin{itemize}
\item If $|T'_4| = 5$, then $T_4 \cong H_{10}$, and clearly $m_{T}(1) = m_{T'_4}(1) = 0 < \frac{n-8}{4}$.
\item If $|T'_4| = 6$, then $T_4 \cong H_{11}$, and $m_{T}(1) = m_{T'_4}(1) = 0 < \frac{n-8}{4}$.
\item If $|T'_4| \ge 7$ and $T'_4 \cong G_4$ or $G_5$, then
\[
m_{T}(1) = m_{T'_4}(1) = \frac{n(T'_4)-7}{4} < \frac{n-8}{4}.
\]
\item If $|T'_4| \ge 7$ and $T_4 \not\cong G_4, G_5$, by the induction hypothesis,
\[
m_{T}(1)=m_{T'_4}(1) \le \frac{n(T'_4)-8}{4} < \frac{n-8}{4}.
\]
\end{itemize}

\begin{figure}[H]
  \centering
  \includegraphics[width=0.7\linewidth]{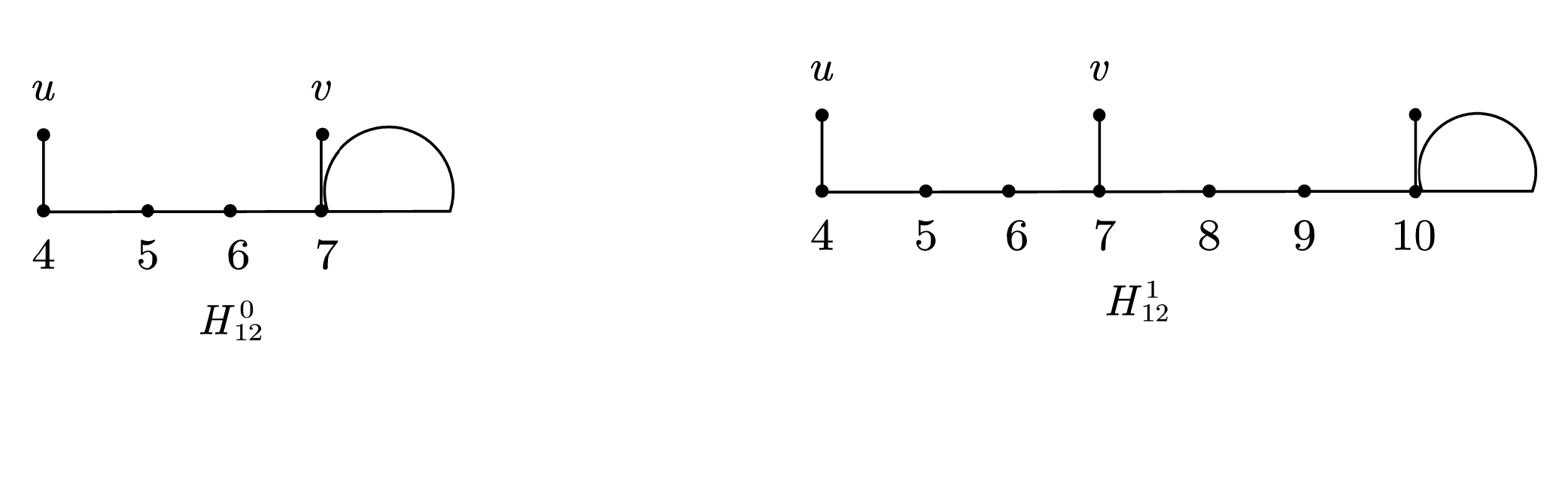}
   \caption{$H^0_{12}$ and $H^1_{12}$} 
\end{figure}
\textbf{Case 1.2.2.3.} If $T'_4$ is not reduced, then $T \cong H^0_{12}$. 
By Lemma 2.3,
\[
m_{T'_4}(1) = 1 + m_{T_7}(1),
\]
where $T_7$ is the component of $T - e_{6,7}$ containing $v_7$. Clearly $T_7$ is reduced and $|T_7| \ge 7$.
Next, we examine $T_7$.
\begin{itemize}
\item If $|T_7| = 7$, then $T_7 \cong G_4$, so $m_{T}(1) = 1 + m_{T_7}(1) = 1 < \frac{n-8}{4}$ by $n \ge 16$.
\item If $|T_7| \ge 8$ and $T_7 \cong G_5$, then $m_{T}(1) = 1 + m_{T_7}(1) = 2 < \frac{n-8}{4}$ by $n \ge 20$.
\item If $T_7 \not\cong G_4, G_5$ and $T_7$ contains no pendant $P_3$, then by the induction hypothesis,
\[
m_{T_7}(1) \le \frac{n(T_7)-8}{4},
\]
so
\[
m_{T}(1) = 1 + m_{T_7}(1) \le 1 + \frac{n(T_7)-8}{4} \le 1 + \frac{n-9-8}{4} < \frac{n-8}{4}.
\]
\end{itemize}

\begin{figure}[H]
  \centering
  \includegraphics[width=0.7\linewidth]{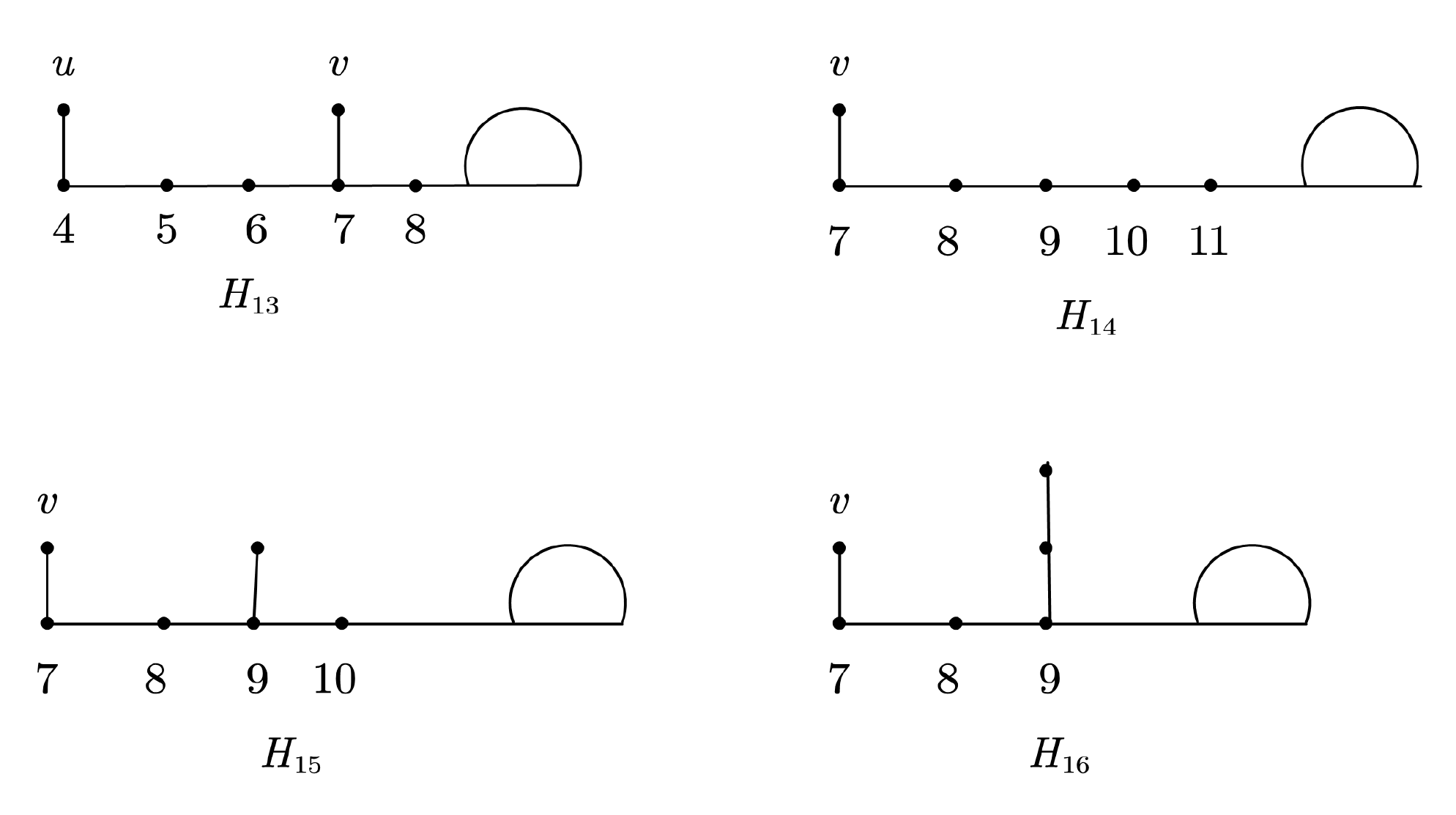}
   \caption{$H_i$ for $13\le i\le 16$} 
\end{figure}

If $T_7$ contains a pendant $P_3$, then $T_4 \cong H_{13}$. Let $T'_7 = T_7 - \{v_7, v_8, v\}$. Clearly $|T'_7| \ge 5$ and $m_{T_7}(1) = m_{L(T'_7)}(1)$ by Lemma 2.2.

Now we discuss $T'_7$:
\begin{itemize}
\item If $|T'_7| = 5$, then $T_7 \cong H_{10}$, so $m_{T}(1) = 1 + m_{T_7}(1) = 1 < \frac{n-8}{4}$ by $n \ge 17$.
\item If $|T'_7| = 6$, then $T_7 \cong H_{11}$, so $m_{T}(1) = 1 + m_{T_7}(1) = 1 < \frac{n-8}{4}$ by $n \ge 18$.
\item If $|T'_7| = 7$, then $T'_7 \cong G_4$, so $m_{T}(1) = 1 + m_{T_7}(1) = 1 + m_{L(T'_7)}(1) = 1 < \frac{n-8}{4}$ by $n \ge 19$.
\item If $|T'_7| \ge 8$ and $T'_7 \cong G_5$, then $m_{T}(1) = 1 + m_{T_7}(1) = 1 + m_{L(T'_7)}(1) = 2 < \frac{n-8}{4}$ by $n \ge 23$.
\item If $|T'_7| \ge 8$, $T'_7 \not\cong G_5$, $T'_7$ has no pendant $P_3$ and is reduced, then by the induction hypothesis,
\[
m_{L(T'_7)}(1) \le \frac{n(T'_7)-8}{4},
\]
so
\[
m_{T}(1) = 1 + m_{L(T'_7)}(1) \le 1 + \frac{n(T'_7)-8}{4} \le 1 + \frac{n-12-8}{4} < \frac{n-8}{4}.
\]
\item If $|T'_7| \ge 8$ and $T'_7$ has a pendant $P_3$, then $T_7 \cong H_{14}, H_{15}$ or $H_{16}$.
\begin{itemize}
\item If $T_7 \cong H_{14}$, let $T'$ be obtained by deleting $v_8, v_9, v_{10}$ and joining $v_7, v_{11}$. As in Case 1.2.1, we get $m_{T}(1) < \frac{n-8}{4}$.
\item If $T_7 \cong H_{15}$ , then $T'_7$ is reduced and $|T'_7| \ge 7$. By Theorem 1.2,
\[
m_{T'_7}(1) < \frac{n(T'_7)-2}{4},
\]
so
\[
m_{T}(1) = 1 + m_{T'_7}(1) < 1 + \frac{n(T'_7)-2}{4} \le 1 + \frac{n-12-2}{4} = \frac{n-8}{4}.
\]
\item If $T \cong H_{16}$, a similar discussion as for $H_{15}$ leads to the same conclusion.
\end{itemize}
\item If $T'_7$ is not reduced, then $T_7 \cong H^{1}_{12}$. Then
\[
m_{T}(1) = 1 + m_{T_7}(1) = 2 + m_{T_{10}}(1),
\]
where $T_{10}$ is the component of $T - e_{9,10}$ containing $v_{10}$.Observe $H^0_{12}$ and $H^1_{12}$. Then, by repeating the same discussion for $T_{10}$ and using symmetry, we find that if $m_{T}(1) \ge \frac{n-8}{4}$, then $T$ must be of the form $G_6$ (see Fig. 2).
By Lemmas 2.3 and 2.7, we easily obtain $m_{T}(1) = l-1$. Note that $n \ge 4(l-1)+14$, so
\[
\frac{n-8}{4} \ge \frac{4(l-1)+14-8}{4} = l-1 + \frac{6}{4} > m_{T}(1).
\]
\end{itemize}

\medskip

\begin{figure}[H]
  \centering
  \includegraphics[width=0.7\linewidth]{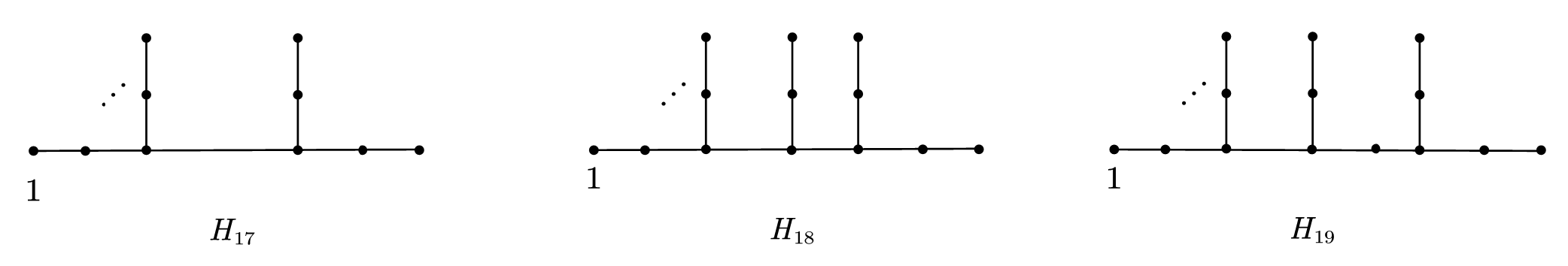}
   \caption{$H_i$ for $17\le i\le 19$} 
\end{figure}

\textbf{Case 1.2.3.} $T \cong H_6$

Clearly $|T_4| \ge 5$. If $|T_4| = 5$, then $T \cong H_{17}$ and $m_{T}(1)=0 < \frac{n-8}{4}$ by Lemma 2.4.

If $|T_4| > 5$, then $|T_4| \ge 8$. When $|T_4| = 8$, $T \cong H_{18}$, so
\[
m_{T}(1) \le 1 + m_{T_3}(1) + m_{T_4}(1) = 1 < \frac{n-8}{4}.
\]
When $|T_4| = 9$, we have $T \cong H_{19}$. Similarly, $m_{T}(1) < \frac{n-8}{4}$.

If $|T_4| \ge 10$, let $T'_4 = T_4 - \{u, v, v_4\}$. Then $|T'_4| \ge 7$ because $|T_4| \ge 10$.

If $T'_4 \cong G_4$, then
\[
m_{T}(1) \le 1 + m_{T_4}(1) = 1 + m_{T'_4}(1) = 1 < \frac{n-8}{4} \quad \text{by } n \ge 15.
\]
For $T'_4 \cong G_5$, we similarly get $m_{T}(1) \le 2 < \frac{n-8}{4}$ by $n \ge 19$.
Next, we consider the case where $|T_4| \ge 10$ and $T‘_4 \ncong G_4$, $G_5$.

\medskip

\begin{figure}[H]
  \centering
  \includegraphics[width=0.7\linewidth]{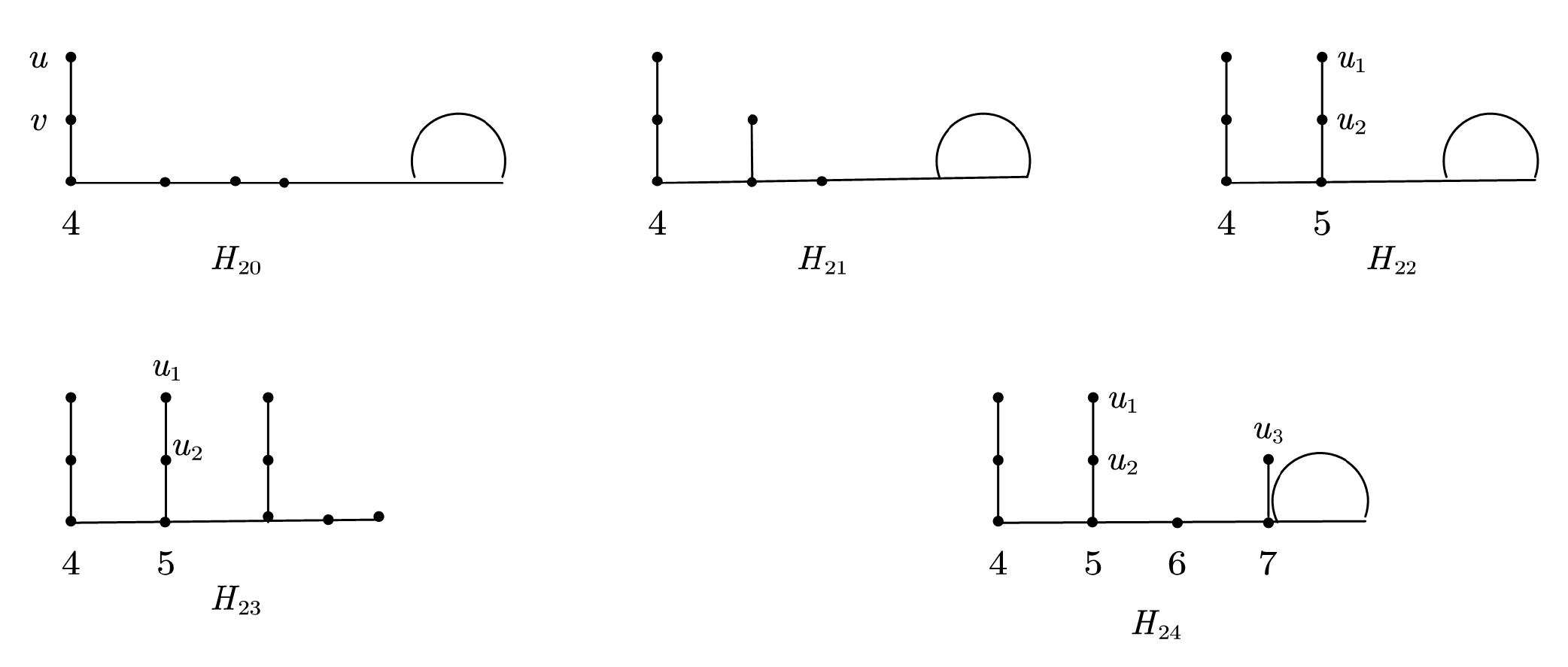}
   \caption{$H_i$ for $20\le i\le 24$} 
\end{figure}

\textbf{Case 1.2.3.1.} If $T'_4$ is reduced, contains no pendant $P_3$, and $T'_4 \not\cong G_4, G_5$, then by the induction hypothesis,
\[
m_{T'_4}(1) \le \frac{n(T'_4)-8}{4},
\]
so
\[
m_{T}(1) \le 1 + m_{T'_4}(1) \le 1 + \frac{n(T'_4)-8}{4} \le 1 + \frac{n-8-8}{4} < \frac{n-8}{4}.
\]

\medskip

\textbf{Case 1.2.3.2.} If $T'_4$ has a pendant $P_3$, then $T_4 \cong H_{20}, H_{21}$ or $H_{22}$.
\begin{itemize}
\item If $T_4 \cong H_{20}$, similarly to Case 1.2.1, we get $m_{T}(1) < \frac{n-8}{4}$.
\item If $T_4 \cong H_{21}$, then $T'_4$ is reduced and $|T'_4| \ge 8$. By Theorem 1.2,
\[
m_{T'_4}(1) < \frac{n(T'_4)-2}{4}.
\]
Applying Lemma 2.8 to $v_5$, we have
\[
m_{T}(1) = m_{\widetilde{T}}(1) + m_{T'_4}(1) < \frac{n(T'_4)-2}{4} \le \frac{n-8-2}{4} < \frac{n-8}{4},
\]
where $\widetilde{T}$ is a graph of the form $T(s_1, s_2)$ shown in Fig. 3.
\item If $T_4 \cong H_{22}$, let $T'_5 = T'_4 - \{v_5, u_1, u_2\}$. Then $|T'_5| \ge 4$ because $|T'_4| \ge 7$. Note $|T'_5| \neq 4$, so $|T'_5| \ge 5$.
\begin{itemize}
\item If $|T'_5| = 5$, then $T_4 \cong H_{23}$ and $m_{T_4}(1)=0$, so
\[
m_{T}(1) \le 1 + m_{T_3}(1) + m_{T_4}(1) = 1 < \frac{n-8}{4}.
\]
\item If $|T'_5| \ge 6$ and $T'_5$ is reduced, then by Theorem 1.2,
\[
m_{T'_5}(1) < \frac{n(T'_5)-2}{4} \le \frac{n-11-2}{4},
\]
so
\[
m_{T}(1) \le 1 + m_{T'_5}(1) < 1 + \frac{n-13}{4} < \frac{n-8}{4}.
\]
\item If $T'_5$ is not reduced, then $T_4 \cong H_{24}$. Apply Lemma 2.7 at $v_7$:
\[
m_{T}(1) = m_{T^*_6}(1) + m_{T_7}(1),
\]
where $T^*_6$ is obtained from $T_6$ (the component of $T-e_{6,7}$ containing $6$) by attaching a pendant $P_2$ at $v_6$, and $T_7$ is the component containing $v_7$. Clearly $m_{T^*_6}(1)=0$. Since $T_7$ is reduced and $|T_7| \ge 7$, by Theorem 1.2,
\[
m_{T_7}(1) < \frac{n(T_7)-2}{4}.
\]
Thus
\[
m_{T}(1) = m_{T^*_6}(1) + m_{T_7}(1) < \frac{n(T_7)-2}{4} \le \frac{n-12-2}{4} < \frac{n-8}{4}.
\]
\end{itemize}
\end{itemize}

\medskip

\begin{figure}[H]
  \centering
  \includegraphics[width=0.7\linewidth]{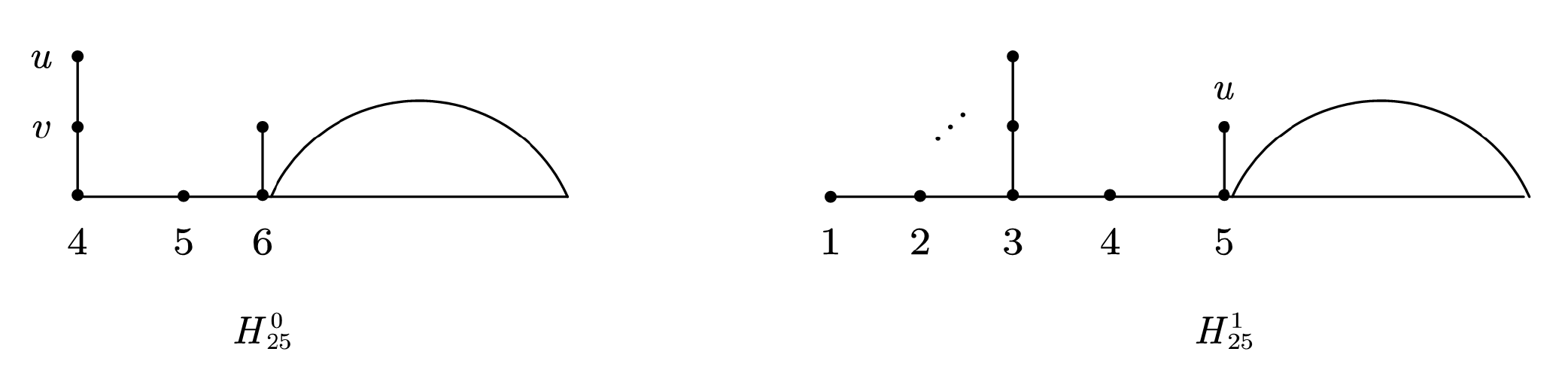}
   \caption{$H^0_{25}$ and $H^1_{25}$} 
\end{figure}

\textbf{Case 1.2.3.3.} If $T'_4$ is not reduced, then $T_4 \cong H^0_{25}$. Apply Lemma 2.7 at $v_6$:
\[
m_{T}(1) = m_{T^*_5}(1) + m_{T'_6}(1),
\]
where $T^*_5$ is obtained from $T'_5$ (the component of $T-e_{5,6}$ containing $v_5$) by attaching a pendant $P_2$ at $v_5$, and $T'_6$ is the component of $T-e_{5,6}$ containing $v_6$.
Clearly $m_{T^*_5}(1)=0$. Since $T'_6$ is reduced and $|T'_6| \ge 7$, by Theorem 1.2,
\[
m_{T'_6}(1) < \frac{n(T'_6)-2}{4}.
\]
Thus
\[
m_{T}(1) = m_{T^*_5}(1) + m_{T'_6}(1) < \frac{n(T'_6)-2}{4} \le \frac{n-9-2}{4} < \frac{n-8}{4}.
\]

\medskip

\textbf{Case 1.3.} If $T_4$ is not reduced, then $T \cong H^1_{25}$. 
Let $T - e_{4,5} = T'_4 \cup T'_5$, where $T'_4$ is the component containing $v_4$ and $T'_5$ is the component containing $v_5$. 
Applying Lemma 2.7 at $v_5$, we obtain
\[
m_{T}(1) = m_{T^*_4}(1) + m_{T'_5}(1),
\]
where $T^*_4$ is obtained from $T'_4$ by attaching a pendant $P_2$ at $v_4$. 
By Lemmas 2.2 and 2.4, we have $m_{T^*_4}(1) = 0$. 
Note that $|T'_5| \ge 7$ and $T'_5$ is reduced. Hence, by Theorem 1.2,
\[
m_{T'_5}(1) < \frac{n(T'_5)-2}{4}.
\]
Therefore,
\[
m_{T}(1) = m_{T^*_4}(1) + m_{T'_5}(1) < \frac{n(T'_5)-2}{4} \le \frac{n-6-2}{4} = \frac{n-8}{4}.
\]

\begin{figure}[H]
  \centering
  \includegraphics[width=0.8\linewidth]{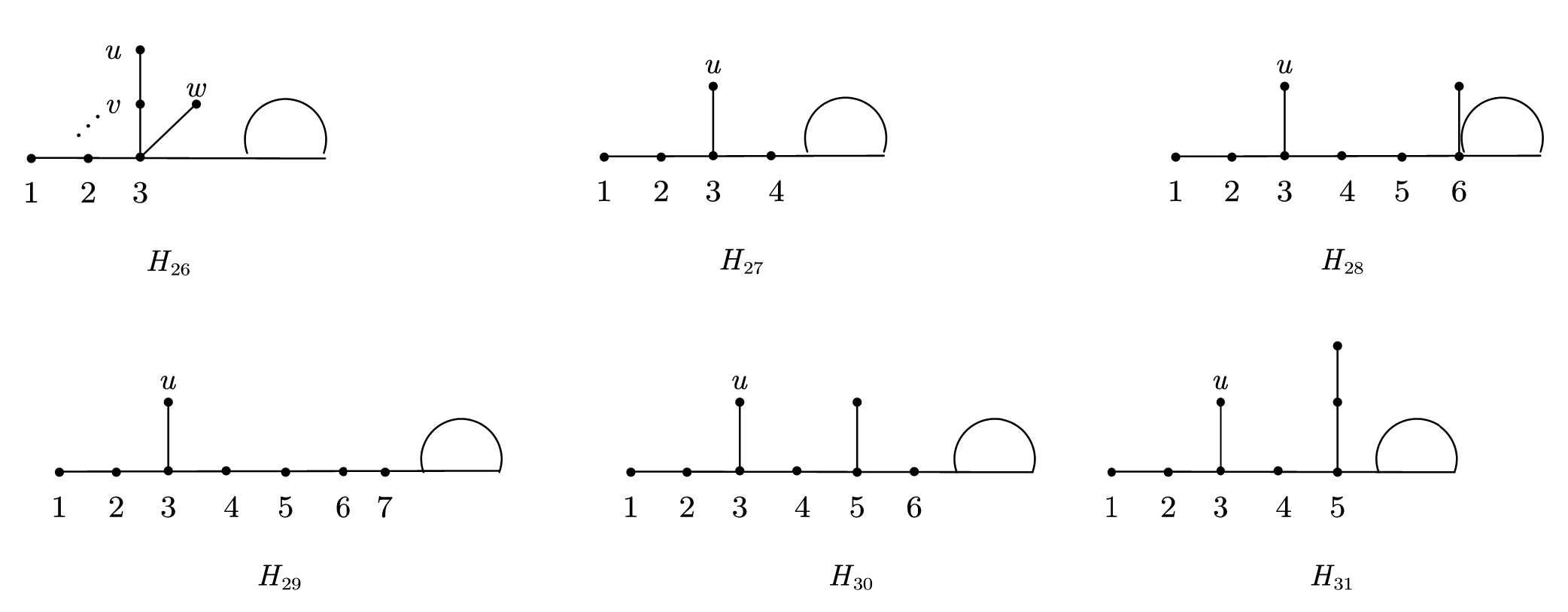}
   \caption{$H_{i}$ for $26\le i\le 31$} 
\end{figure}
\textbf{Case 2.} $v_3$ or $v_{d-1}$ is a quasi-pendant vertex. Without loss of generality, assume $v_3$ is quasi-pendant.

\medskip

\textbf{Case 2.1.} $d_{v_3} \ge 4$.

Then $v_3$ has at least two pendant $P_2$, as shown in $H_{26}$. By Lemma 2.8,
\[
m_{T}(1) = m_{T-\{u,v\}}(1).
\]
Since $T-\{u,v\}$ is reduced and contains no pendant $P_3$, by Theorem 1.1,
\[
m_{T-\{u,v\}}(1) \le \frac{n(T-\{u,v\})-6}{4}.
\]
Hence
\[
m_{T}(1) = m_{T-\{u,v\}}(1) \le \frac{n(T-\{u,v\})-6}{4} = \frac{n-8}{4}.
\]
Equality holds iff $m_{T-\{u,v\}}(1) = \frac{n(T-\{u,v\})-6}{4}$, which means $T-\{u,v\} \cong G_1$, so $T \cong G_2$.

\medskip

\textbf{Case 2.2.} $d_{v_3} = 3$.

Let $T_3 = T - \{v_1, v_2\}$. By Lemma 2.8, $m_{T}(1) = m_{T_3}(1)$. Clearly $T_3$ is reduced.

\textbf{Case 2.2.1.} If $T_3$ contains no pendant $P_3$, then by Theorem 1.1,
\[
m_{T_3}(1) \le \frac{n(T_3)-6}{4},
\]
so
\[
m_{T}(1) = m_{T_3}(1) \le \frac{n(T_3)-6}{4} = \frac{n-8}{4}.
\]
Equality holds iff $m_{T_3}(1) = \frac{n(T_3)-6}{4}$, which means $T_3 \cong G_1$, so $T \cong G_3$.

\textbf{Case 2.2.2.} If $T_3$ contains a pendant $P_3$, then $T \cong H_{27}$. Let $T'_3 = T_3 - \{u, v_3, v_4\}$. Now consider $T'_3$.

\textbf{Case 2.2.2.1.} If $T'_3$ is not reduced, then $T \cong H_{28}$. Let $T_4$ be the component of $T - e_{3,4}$ containing $v_4$. Note $|T_4| \ge 8$ because $n \ge 12$. $T_4$ is reduced, $T_4 \not\cong G_1, G_4, G_5$, and $T_4$ has no pendant $P_3$. By the induction hypothesis,
\[
m_{T_4}(1) \le \frac{n(T_4)-8}{4},
\]
with equality iff $T_4 \cong G_2$ or $G_3$. By Lemma 2.6,
\[
m_{T}(1) = 1 + m_{T_4}(1) \le 1 + \frac{n(T_4)-8}{4} = 1 + \frac{n-4-8}{4} = \frac{n-8}{4},
\]
with equality iff $T_4 \cong G_2$ or $G_3$, i.e., $T \cong G_2$ or $G_3$.

\textbf{Case 2.2.2.2} If $T'_3$ is reduced, we consider whether $T'_3$ contains a pendant $P_3$.
\begin{itemize}
\item If $T'_3$ contains no pendant $P_3$, since $|T'_3| \ge 7$ (because $n \ge 12$), by Theorem 1.1,
\[
m_{T'_3}(1) \le \frac{n(T'_3)-6}{4}.
\]
Thus
\[
m_{T}(1) = m_{T_3}(1) = m_{T'_3}(1) \le \frac{n(T'_3)-6}{4} = \frac{n-5-6}{4} < \frac{n-8}{4}.
\]
\item If $T'_3$ contains a pendant $P_3$, then $T$ is isomorphic to $H_{29}, H_{30}$ or $H_{31}$. In each case, similar to the previous discussions, we obtain $m_{T}(1) < \frac{n-8}{4}$.

\end{itemize}
\end{proof}
\begin{proof}[Proof of Theorem 1.4]
If $T \not\cong G_4$ or $G_5$, then by Theorem 1.3 we have
\[
m_{T}(1) \le \frac{n-8}{4} < \frac{n-7}{4}.
\]
Thus we only need to consider the case $T \cong G_4$ or $G_5$. By using a computer, we easily obtain $m_{L(G_4)}(1) = 0$ and $m_{L(G_5)}(1) = 1$. Hence, when $T \cong G_4$ or $G_5$,
\[
m_{T}(1) = \frac{n-7}{4}.
\]

This completes the proof.
\end{proof}

\section{Acknowledgments}

We gratefully acknowledge the support of the Graduate Innovation Program of China University of Mining and Technology (2025WLKXJ146), the Fundamental Research Funds for the Central Universities, and the Postgraduate Research and Practice Innovation Program of Jiangsu Province (KYCX25\_2858) for this work.

\end{document}